\documentclass[11pt]{amsart}

\usepackage[utf8]{inputenc}
\usepackage[T1]{fontenc}
\usepackage{lmodern}
\usepackage{microtype}

\usepackage{amsmath,amssymb,amsfonts,amsthm,mathtools}
\usepackage{enumitem}
\usepackage[margin=1.15in]{geometry}

\usepackage[colorlinks=true,linkcolor=blue,citecolor=blue,urlcolor=blue]{hyperref}
\usepackage{aliascnt}
\usepackage[nameinlink,capitalize]{cleveref}

\newtheorem{theorem}{Theorem}[section]

\newaliascnt{proposition}{theorem}
\newtheorem{proposition}[proposition]{Proposition}
\aliascntresetthe{proposition}

\newaliascnt{lemma}{theorem}
\newtheorem{lemma}[lemma]{Lemma}
\aliascntresetthe{lemma}

\newaliascnt{corollary}{theorem}
\newtheorem{corollary}[corollary]{Corollary}
\aliascntresetthe{corollary}

\theoremstyle{definition}

\newaliascnt{definition}{theorem}

\aliascntresetthe{definition}

\theoremstyle{remark}

\newaliascnt{remark}{theorem}

\aliascntresetthe{remark}

\crefname{theorem}{Theorem}{Theorems}
\crefname{proposition}{Proposition}{Propositions}
\crefname{lemma}{Lemma}{Lemmas}
\crefname{corollary}{Corollary}{Corollaries}
\crefname{definition}{Definition}{Definitions}
\crefname{remark}{Remark}{Remarks}
\crefname{equation}{equation}{equations}
\Crefname{equation}{Equation}{Equations}

\newcommand{\bbP}{\mathbb P}
\newcommand{\bbE}{\mathbb E}

\newcommand{\one}{\mathbf 1}

\newcommand{\calU}{\mathcal U}
\newcommand{\calY}{\mathcal Y}
\newcommand{\calX}{\mathcal X}

\newcommand{\vol}{\operatorname{vol}}
\newcommand{\rank}{\operatorname{rank}}
\newcommand{\tors}{\operatorname{tors}}
\newcommand{\inj}{\operatorname{inj}}
\newcommand{\Inj}{\operatorname{InjRad}}
\newcommand{\dd}{\mathop{}\!\mathrm{d}}

\title[Sparse Random Covers and First Homology Growth]{Sparse Random Covers and Growth of Torsion in First Homology}
\author{Raz Slutsky}
\address{Merton College, Merton Street, OX1 4JD, Oxford, United Kingdom}
\email{raz.slutsky@maths.ox.ac.uk}

\begin{document}

\begin{abstract}
We construct random open covers of higher-rank locally symmetric spaces using a construction we call scaffolded Poisson processes. Let \(X=G/K\) be a
symmetric space of noncompact type and real rank \(r\ge2\).  We prove a general vanishing theorem for the normalized torsion in first homology along sequences of torsion-free lattices in $G$. In particular, if $G$ is simple, we get
\[
\dfrac{\log |H_1(M_n;\mathbb{Z})_{\tors}|}{\mathrm{vol}(M_n)} \longrightarrow 0
\]
for any sequence of distinct manifolds $M_n = \Gamma_n \backslash X$.
This answers a question of Ab\'ert, Gelander, and Nikolov,
and confirms the degree-one vanishing with trivial integral coefficients predicted by a conjecture of Bergeron and Venkatesh in the higher-rank setting. In addition, we get quantitative bounds with respect to the minimal injectivity radius for both the torsion in first homology and the minimal number of generators of $\Gamma$. Finally, we prove the analogous statements for affine buildings.
\end{abstract}
\maketitle

\section{Introduction}

Let \(X=G/K\) be a symmetric space of noncompact type
without Euclidean factors, where \(G=\operatorname{Isom}(X)^\circ\) is a semisimple Lie group and \(K<G\) 
is a maximal compact subgroup. 
We write \(r=\rank_{\mathbb R}X\) for the real rank and assume \(r\ge2\). We denote by $M=\Gamma\backslash X$ a finite-volume quotient of $X$, where $\Gamma$ is a lattice in $G$. Typical examples include the locally symmetric orbifold
\(
 \mathrm{SL}_3(\mathbb Z)\backslash
 \mathrm{SL}_3(\mathbb R)/\mathrm{SO}(3)
\)
and compact quotients of
\(\mathbb H^2\times\mathbb H^2\). For such a lattice,
we denote \(d(\Gamma)\) to be the minimal size of a
generating set. 

This paper concerns the torsion part of the first homology \(H_1(\Gamma;\mathbb Z)\). While the growth of rational homology of these spaces is relatively well-understood, the integral homology is far more mysterious and the subject of various conjectures, as will be discussed later.

 For a finite-volume quotient
\(M=\Gamma\backslash X\), let \(M_{<L}\) denote the set of points whose injectivity radius is less than \(L\).  We say that a sequence \(M_n = \Gamma_n \backslash X \) Benjamini--Schramm converges to \(X\) if, for every
fixed \(L>0\),
\[
 \frac{\vol((M_n)_{<L})}{\vol(M_n)}\longrightarrow0.
\]

Our main result is the following theorem.

\begin{theorem}\label{thm:intro-bs-h1-torsion}
Let \(M_n=\Gamma_n\backslash X\) be a sequence of finite-volume irreducible
\(X\)-manifolds.  If \(M_n\) BS-converges to \(X\), then
\[
 \frac{\log |H_1(M_n;\mathbb Z)_{\tors}|}{\vol(M_n)}
 \longrightarrow0.
\]
\end{theorem}

This gives, to our knowledge, the first general vanishing
theorem for normalized first-homology torsion of cocompact higher-rank
lattices.  The lattices need not be nested or lie in a fixed commensurability
class. This answers a question of Ab\'ert, Gelander and Nikolov \cite[\S8]{AbertGelanderNikolov}, and confirms the degree-one case with integral coefficients of the conjecture of Bergeron and Venkatesh \cite[Conj. 1.3]{BergeronVenkatesh}. We emphasise that this torsion statement does not follow from the corresponding results on the number of generators. While $d(\Gamma)$ controls the dimension of $H_1(\Gamma;\mathbb F)$, it gives no control on the order of $H_1(\Gamma;\mathbb Z)_{\tors}$, as even a one-generated abelian group can have arbitrarily large torsion.

We also note that if \(G\) is simple, the Benjamini--Schramm convergence hypothesis in the above theorem, and in the subsequent results for locally symmetric spaces, is automatic along every sequence of pairwise non-conjugate lattices, and hence along every sequence of pairwise non-isometric quotients, by \cite{AbertBergeronBiringerGelanderNikolovRaimbaultSamet}.

Our next result is a quantitative version in terms of the injectivity radius.
For a manifold \(M\), write \(\inj_x(M)\) for the
injectivity radius at \(x\), and put
\(\Inj(M)=\inf_{x\in M}\inj_x(M)\). For a compact manifold, this is equal to half the length of the shortest closed geodesic, also called the \emph{systole}.

\begin{theorem}\label{thm:main}
There exist constants \(C_X,R_X<\infty\), depending only on \(X\), such that
every torsion-free lattice \(\Gamma<G\), with \(M=\Gamma\backslash X\) and
\(R=\Inj(M)\ge R_X\), satisfies
\[
 \log |H_1(M;\mathbb Z)_{\tors}|
 \le C_X\vol(M)R^{(1-r)/2}(\log R)^2.
\]
The same bound holds for $d(\Gamma)$.
\end{theorem}

The essential point is that
\(R^{(1-r)/2}(\log R)^2\to0\) as \(R\to\infty\).  The generator estimate
answers a question of Fr\k{a}czyk, Mellick, and Wilkens up to the logarithmic factor \cite[Question 6]{FraczykMellickWilkens}.

Lapan, Linowitz, and Meyer proved that in cocompact congruence towers, the systole grows logarithmically in the volume
\cite{LapanLinowitzMeyer}.  Hence, for such towers, one obtains an estimate of the form
\[
 \max\bigl\{d(\pi_1M),\log |H_1(M;\mathbb Z)_{\tors}|\bigr\}
 \le C\vol(M)(\log\vol(M))^{(1-r)/2}
       (\log\log\vol(M))^2.
\]

For the minimal number of generators, we can remove the torsion-freeness assumption.

\begin{theorem}\label{thm:intro-bs-orbifold}
Let \(O_n=\Gamma_n\backslash X\) be a sequence of finite-volume quotients by
irreducible lattices, not assumed torsion-free.  If \(O_n\) BS-converges to
\(X\), then
\[
 d(\Gamma_n)=o(\vol(O_n)).
\]
\end{theorem}

This extends
\cite{FraczykMellickWilkens} by removing torsion-freeness and thereby settles
the remaining orbifold case of a conjecture of Ab\'ert, Gelander, and Nikolov
\cite[Conj.~3]{AbertGelanderNikolov}.
As an immediate consequence, first-homology growth vanishes simultaneously
over all fields.

\begin{corollary}\label{cor:intro-modp}
Under the hypotheses of \cref{thm:intro-bs-orbifold},
\[
 \sup_{\mathbb F\;\mathrm{a\ field}}
 \frac{\dim_{\mathbb F}H_1(\Gamma_n;\mathbb F)}
      {\vol(\Gamma_n\backslash X)}
 \longrightarrow0.
\]
In particular, for every prime \(p\),
\(\dim_{\mathbb F_p}H_1(\Gamma_n;\mathbb F_p)=
o(\vol(\Gamma_n\backslash X))\).
\end{corollary}

Consequently, for torsion-free BS-convergent sequences, the
volume-normalized generator rank, first-homology dimensions over every field,
and logarithmic first-homology torsion all vanish. This proves the
degree-one case of a conjecture of Li, L\"oh, Moraschini, Sauer, and Uschold
for torsion-free lattices \cite[Conj.~1.15]{sauer_etal} in this setting.  

These results fit into the
``below-rank'' phenomenon, according to which invariants of higher-rank
lattices are expected to vanish in all degrees below the rank; see Bader and
Sauer \cite{bader_sauer}. For the number-theoretic significance of torsion in arithmetic
(co)homology, particularly its relation to mod-\(p\) Galois
representations, see
\cite{BergeronECM,ScholzeTorsionCohomology}, and the introduction of \cite{AbertBergeronFraczykGaboriau}.

We now review some related results. For rational Betti numbers,
asymptotic questions of this kind are governed by a powerful approximation
theory.  L\"uck's approximation theorem identifies normalized rational
Betti-number growth in residual towers with \(L^2\)-Betti numbers
\cite{LuckApprox,LuckSurvey}.  In the locally symmetric setting, the work of
Ab\'ert, Bergeron, Biringer, Gelander, Nikolov, Raimbault, and Samet showed that
Benjamini--Schramm convergence is the appropriate geometric analogue and
established corresponding approximation results for normalized rational Betti
numbers \cite{AbertBergeronBiringerGelanderNikolovRaimbaultSamet}.  By
contrast, no comparably general approximation theory is known for integral torsion.  In arithmetic settings,
Bergeron and Venkatesh formulated a precise conjectural framework for the asymptotic growth
of homology torsion in congruence towers \cite{BergeronVenkatesh}.  In degree
one and higher rank, their conjectural picture predicts vanishing of normalized
logarithmic torsion in the relevant towers, but general geometric mechanisms
proving such vanishing have been scarce.

These questions fit into a broader theme: the volume of a
locally symmetric space should control its topological and algebraic
complexity.  By now, this is a long-established phenomenon in various settings; see \cite{BaderGelanderSauer,
GromovBetti,
FraczykModTwo,
GelanderSlutsky, 
GelanderSlutsky2,
KaponSlutsky}.

Gelander proved that, for fixed \(X\), the minimal number of generators of an
irreducible lattice satisfies
\[
d(\Gamma)=O_X\bigl(\vol(\Gamma\backslash X)\bigr)
\]
\cite{GelanderRank}.  In higher rank, Ab\'ert, Gelander, and
Nikolov conjectured that generator growth should in fact be sublinear
\cite[Conj.~3]{AbertGelanderNikolov}. They proved the vanishing of rank gradient for sequences of subgroups of certain
higher-rank lattices which they call right-angled lattices.  In the nonuniform
case, writing \(v=\vol(\Gamma\backslash X)\), Lubotzky and the author proved
the stronger bound \(d(\Gamma)=O_X(\log v)\)
\cite{LubotzkySlutsky}.  More recently, in a breakthrough of Fr\k{a}czyk, Mellick, and Wilkens, they 
proved sublinear growth of the number of generators for torsion-free
Benjamini--Schramm convergent sequences of lattices, using Poisson--Voronoi
tessellations on the symmetric space \cite{FraczykMellickWilkens}.  Their work
also settled the fixed-price conjecture for higher-rank lattices, see also our recent generalization \cite{slutsky2026product}.  Since
\(d(\Gamma)\) bounds \(\dim_{\mathbb F}H_1(\Gamma;\mathbb F)\) for every field
\(\mathbb F\), it also gives the corresponding vanishing of first-homology
growth over arbitrary fields.  See also the work of Ab\'ert and Mellick
connecting cost, point processes, and rank gradient \cite{AbertMellick}.

For first-homology torsion, Ab\'ert, Gelander, and Nikolov also proved
vanishing along sequences of subgroups of right-angled lattices
\cite{AbertGelanderNikolov}. We note that if \(\Gamma\) is an arithmetic lattice with the
generalized congruence subgroup property, then for subgroups of this fixed lattice, torsion in first homology grows at most polynomially in the index, see
\cite[\S5.1]{AbertGelanderNikolov}. For homology in positive characteristic,
Fr\k{a}czyk proved a quantitative sublinear bound for first mod-\(2\) homology
of torsion-free higher-rank lattices \cite{FraczykModTwo}.  A major result in
the nonuniform arithmetic setting is the work of Ab\'ert, Bergeron,
Fr\k{a}czyk, and Gaboriau.  Using a homotopical method called effective
rebuilding, they proved vanishing of homology over finite fields and of
integral homology torsion in all degrees less than the rank,
for Farber sequences of finite-index subgroups of a fixed nonuniform arithmetic
lattice.  They also obtained quantitative estimates for principal
congruence subgroups \cite{AbertBergeronFraczykGaboriau}.

Finally, we prove an affine-building analogue of our main results.

\begin{theorem}\label{intro:buildings}
Let \(\mathcal B\) be a locally finite thick Euclidean building of Euclidean rank
\(r\ge2\)  and let \(Y_n=\Gamma_n\backslash\mathcal B\) be compact quotients by torsion-free
uniform lattices \(\Gamma_n<\operatorname{Aut}(\mathcal B)\). If \(Y_n\) BS-converges to \(\mathcal B\), then
\[
d(\Gamma_n)=o(\mu(Y_n)),\qquad
        \log |H_1(\Gamma_n;\mathbb Z)_{\tors}|=o(\mu(Y_n)).
\]
\end{theorem}

This applies both to semisimple groups over local fields and to exotic affine buildings. Quantitative versions also apply; see \Cref{sec:affine-buildings}.

For Bruhat--Tits buildings, the BS-convergence assumption is automatic in
many higher-rank algebraic situations.  More precisely, the results of
Gelander and Levit imply the BS-convergence of pairwise nonconjugate irreducible lattices in
 non-Archimedean semisimple analytic groups with property~\((T)\) and
rank at least two
\cite{GelanderLevit}.
As far as we are aware, except for the number of generators in products of trees in \cite{FraczykMellickWilkens}, and the $\mathbb{F}_2$-betti numbers in \cite{fraczyk2022growth}, this appears to be new in all other cases.

\paragraph{\textbf{Overview of the Proof}} The proofs are inspired by the work of Fr\k{a}czyk, Mellick and Wilkens,
together with constructions of Gelander. Our goal is to construct an open cover of the manifold by contractible balls, such that the total number of pairwise intersections is sub-linear in the volume. This gives an efficient topological model for the manifold in degree one by taking the nerve of such a cover, see \cref{lem:topology-from-nerve}. To do so, we start with a deterministic maximal 1-separated set of points on the manifold, which we call the scaffold, and then take the set of balls of radius $2$ around them. This covers the entire manifold and detects the topology, but it has
linearly many intersections in terms of the volume. Therefore, we would like to retain only a small proportion of these balls. To do so, we sample a random set of points which is sparse, on average, using a Poisson point process, and take large balls around these random points. The expected number of nearby pairs of the random points is small.
By itself, however, it may miss parts of the manifold. In these missed regions, we add the small, radius-2 balls around the scaffold. We thus get a combination of the good properties of the two covers:  it covers the entire manifold, and also has sub-linearly many intersections in terms of the volume.

The rest of the argument is designed to make this idea work. There is a tension between the intensity of the point process, which controls how sparse the random points are, and the radius of the large balls around them. These need to be chosen in such a way that the missed scaffold points are rare enough, while the number of intersections between the large balls is kept small. To control this, we use volume estimates of balls in higher rank symmetric spaces. Interestingly, this is the only geometric input in the proof.

For the BS-convergent theorems, the Poisson process is sampled only in a
thick region, while the thin part is handled
deterministically. BS convergence makes the contribution of the thin part negligible.
For this, we use results of
Fr\k{a}czyk--Hurtado--Raimbault in the compact case \cite{FraczykHurtadoRaimbaultTopComplexity}, and of Gelander in the noncompact case
\cite{GelanderHomotopy}, which allow us to estimate this contribution.

\paragraph{\textbf{Organization of the paper.}}
In \cref{sec:preliminaries} we collect the ball-volume estimates, Poisson
process facts, and topological tools used throughout.  In
\cref{sec:thick-case} we construct the scaffolded Poisson cover in the
large-injectivity-radius setting and prove \cref{thm:main}. These two sections are the heart of the proof.  In
\cref{sec:bs-torsion-manifolds} we adapt the construction to thick regions of
BS-convergent manifolds and prove \cref{thm:intro-bs-h1-torsion}, treating the
noncompact and compact cases separately.  In \cref{sec:core-scaffolds} we
construct a bounded-degree scaffold for a general lattice,
and in \cref{sec:bs-rank} we use it to prove the orbifold rank theorem, 
\cref{thm:intro-bs-orbifold} and \cref{cor:intro-modp}.  Finally,
\cref{sec:affine-buildings} establishes the corresponding results for affine
buildings.

\paragraph{\textbf{Acknowledgements.}}
I am very grateful to Miko{\l}aj Fr\k{a}czyk, Gil Goffer, Ido Grayevsky, Elyasheev Leibtag,
Konstantin Recke and Roman Sauer for reading and commenting on earlier
versions of this paper.  I also thank Shaked Bader and Lawk Mineh for useful
conversations.  Finally, I thank Tsachik Gelander for his continued support and
mentorship.

\paragraph{\textbf{Statement on AI usage.} ChatGPT 5.5 was used to locate certain references related to the affine building case, to assist in simplifying the proof of Lemma \ref{lem:log-radii}, and to suggest edits for clarity. All mathematical ideas and arguments are the author’s own.}

\section{Preliminaries}\label{sec:preliminaries}

Fix a basepoint \(o\in X\) and write
\(v(s)=\vol B_X(o,s)\).  By homogeneity, this is independent of \(o\).  Put
\(m=r-1\).  If \(Q=\Gamma\backslash X\) is a manifold or orbifold quotient,
then \(\vol B_Q(x,s)\le v(s)\) for every \(x\in Q\); equality holds whenever
the radius-\(s\) ball at \(x\) is embedded.  All logarithms are natural.

\subsection{Volume growth and the choice of radius}

\begin{lemma}[Ball-volume asymptotics]\label{lem:volume}
There are constants \(a_X,A_X,h>0\) such that, for every \(s\ge1\),
\[
 a_X e^{hs}s^{m/2}\le v(s)\le A_X e^{hs}s^{m/2}.
\]
Here \(h\) is the volume entropy of \(X\).
\end{lemma}

\begin{proof}
This is standard.  In the notation of
\cite[Proposition~4]{LeuzingerT}, if \(\rho_X\) is the half-sum of the positive
restricted roots, counted with multiplicity, then
\(v(s)\asymp_X s^{(r-1)/2}e^{2\|\rho_X\|s}\) for \(s\ge1\).  Thus one may
take \(h=2\|\rho_X\|\), and \(m=r-1\).
\end{proof}

Set \(Q_0=m/2+4\).  Choose once and for all \(c_X\ge0\) so large that
\((a_X/A_X)e^{hc_X}\ge Q_0\), and, for \(t\ge e\), define
\[
 \rho(t)=t+\frac1h\log\log t+c_X.
\]
The same function \(\rho\) will be used throughout the paper.

\begin{lemma}\label{lem:log-radii}
For every \(B_0\ge0\), there are constants
\(T_{X,B_0},C_{X,B_0}<\infty\), depending only on \(X\) and \(B_0\), such
that, for every \(t\ge T_{X,B_0}\) and every \(B\in[0,B_0]\),
\begin{align}
 \exp\left(-\frac{v(\rho(t))}{v(t)}\right)
 &\le t^{-Q_0},                                      \label{eq:miss-prob}\\
 \frac{v(\rho(t)+B)}{v(t)}
 &\le C_{X,B_0}\log t,                               \label{eq:rho-volume}\\
 \frac{v(2\rho(t)+B)}{v(t)^2}
 &\le C_{X,B_0}t^{-m/2}(\log t)^2.                   \label{eq:prox-volume}
\end{align}
\end{lemma}

\begin{proof}
Choose \(T_{X,B_0}\ge e\) so large that, for every
\(t\ge T_{X,B_0}\),
\[
 \frac1h\log\log t+c_X+B_0\le t,
 \qquad
 2\left(\frac1h\log\log t+c_X\right)+B_0\le t.
\]
Because \(c_X\ge0\) and \(t\ge e\), the definition gives \(\rho(t)\ge t\).
The choice of \(T_{X,B_0}\) also gives
\(\rho(t)+B\le2t\) and \(2\rho(t)+B\le3t\), uniformly for
\(B\in[0,B_0]\).

Since \(e^{h(\rho(t)-t)}=e^{hc_X}\log t\), \cref{lem:volume} gives
\[
 \frac{v(\rho(t))}{v(t)}
 \ge \frac{a_X}{A_X}e^{h(\rho(t)-t)}
      \left(\frac{\rho(t)}{t}\right)^{m/2}
 \ge Q_0\log t.
\]
This proves \eqref{eq:miss-prob}.  The same comparison gives
\[
 \frac{v(\rho(t)+B)}{v(t)}
 \le \frac{A_X}{a_X}e^{h(c_X+B_0)}2^{m/2}\log t
\]
and
\[
 \frac{v(2\rho(t)+B)}{v(t)^2}
 \le \frac{A_X}{a_X^2}e^{h(2c_X+B_0)}3^{m/2}
       t^{-m/2}(\log t)^2.
\]
Taking \(C_{X,B_0}\) to be the larger of the two displayed coefficients
proves the remaining assertions.
\end{proof}

Whenever \(B_0\) is fixed in terms of \(X\), we absorb
\(T_{X,B_0}\) and \(C_{X,B_0}\) into constants denoted by \(T_X\) and
\(C_X\).

We shall repeatedly use the following elementary lifting observation.  Let
\(M=\Gamma\backslash X\) be a manifold, and suppose that a finite family of
embedded balls \(B_M(c_j,r_j)\), \(0\le j\le k\), has nonempty intersection.
If
\begin{equation}\label{eq:lifting}
 \inj_{c_0}(M)>r_0+r_j\qquad\text{for every }j,
\end{equation}
then, after choosing a lift of \(c_0\), there is a unique compatible lift of
each of the other balls, and the intersection downstairs is isometric to the
intersection of these lifted balls in \(X\).  Since \(X\) is a Hadamard
manifold and its metric balls are convex, the intersection is contractible.

\subsection{Poisson point processes}

We use Poisson point processes with respect to Riemannian volume. Let \(Y\) be a finite-volume Borel subset of a
manifold or orbifold, and let \(\lambda>0\).  A Poisson point process
\(\calY\) on \(Y\) of intensity \(\lambda\) is characterized by the property
that
\[
 N_{\calY}(A):=\#(\calY\cap A)
\]
has Poisson distribution with mean \(\lambda\vol(A)\) for every Borel set
\(A\subset Y\), and the random variables
\(N_{\calY}(A_1)\) and $N_{\calY}(A_2)$ are independent whenever the sets
\(A_1,A_2\) are disjoint.  The process is finite and simple
almost surely.  In our settings, these processes are very easy to describe. One may first sample
\(N\sim\operatorname{Pois}(\lambda\vol(Y))\) and then, conditional on \(N\),
sample \(N\) points independently according to the normalized volume on \(Y\).

The following is the two-point case of the multivariate Mecke formula
\cite[Theorem~4.4]{LastPenrose}.

\begin{lemma}[Mecke formula]\label{lem:mecke}
Let \(\calY\) be a Poisson point process of intensity \(\lambda\) on a
finite-volume Borel set \(Y\).  If
\(f:Y\times Y\to[0,\infty)\) is measurable and symmetric, then
\[
 \bbE\sum_{\{y,z\}\subset\calY} f(y,z)
 =\frac{\lambda^2}{2}\int_Y\int_Y f(y,z)\,\dd y\,\dd z,
\]
where the sum is over unordered pairs of distinct points of \(\calY\).
\end{lemma}

\subsection{Graphs, good covers, and Gabber's estimate}

We first record an elementary observation for the rank argument.

\begin{lemma}\label{lem:graph-surjection}
Let \(\widehat{\mathcal G}\) be a connected graph.  Suppose that \(\Gamma\)
acts freely on its geometric realization and that the quotient graph
\(\widehat{\mathcal G}/\Gamma\) is finite.  Then
\[
 d(\Gamma)\le |E(\widehat{\mathcal G}/\Gamma)|.
\]
\end{lemma}

\begin{proof}
The quotient map is a regular graph covering with deck group \(\Gamma\), so
\(\pi_1(\widehat{\mathcal G}/\Gamma)\) surjects onto \(\Gamma\).  If \(Q\) is
a finite connected graph, then
\(\rank\pi_1(Q)=|E(Q)|-|V(Q)|+1\le |E(Q)|\), which proves the claim.
\end{proof}

Let \(\calU\) be a cover of a space \(M\).  Its \emph{nerve}
\(N(\calU)\) is the simplicial complex with one vertex for each
\(U\in\calU\), and with a simplex \(\{U_0,\ldots,U_k\}\) whenever
\(U_0\cap\cdots\cap U_k\ne\varnothing\).  A \emph{good cover} is an open
cover for which every nonempty finite intersection is contractible.  The nerve
lemma states that the nerve of a good cover is homotopy equivalent to the
covered space \cite[Cor.~4G.3]{Hatcher}.

We shall also use Gabber's lemma in the following form; see
\cite[Lemma~2.2]{EmeryTorsionHomology}.  If \(K\) is a finite connected
\(2\)-dimensional CW-complex with \(n_1\) unoriented \(1\)-cells, and every
cellular boundary vector of a \(2\)-cell has Euclidean norm at most \(L\), then
\[
 \log |H_1(K;\mathbb Z)_{\tors}|\le n_1\log L.
\]

\begin{lemma}\label{lem:topology-from-nerve}
Let \(M\) be a connected manifold and let \(\calU\) be a finite good cover of
\(M\).  If the nerve \(N(\calU)\) has \(E\) edges, then
\[
 d(\pi_1(M))\le E,
 \qquad
 \log |H_1(M;\mathbb Z)_{\tors}|\le E\log \sqrt{3}.
\]
\end{lemma}

\begin{proof}
By the nerve lemma, \(N(\calU)\) is homotopy equivalent to \(M\).  Since the
nerve is connected, its one-skeleton surjects on fundamental groups, and the
first assertion follows from \cref{lem:graph-surjection}, or directly from the
rank formula for a finite connected graph.

First homology is computed by the two-skeleton.  Every two-simplex has cellular
boundary vector of Euclidean norm \(\sqrt3\).  Gabber's lemma therefore gives
\[
 \log |H_1(M;\mathbb Z)_{\tors}|
 \le E\log\sqrt3 
\]
\end{proof}

\section{The thick case}\label{sec:thick-case}

In this section, \(M=\Gamma\backslash X\) is compact and
\(R=\Inj(M)\).  We prove \cref{thm:main} by constructing a random good cover
whose nerve is sparse on average.

\subsection{The unit-scale scaffold}

Assume \(\Inj(M)>1\), and choose a maximal \(1\)-separated set
\(\calX=\{x_i:i\in I\}\subset M\).  Thus distinct points of \(\calX\) are at
distance at least one, and maximality implies that the open balls
\(B_M(x_i,1)\) cover \(M\).  The balls \(B_M(x_i,1/2)\) are pairwise
disjoint.  We call \(\calX\) the \emph{scaffold}, and its points scaffold
centres.

\begin{lemma}\label{lem:scaffold-counts}
For every fixed \(D<\infty\), there is a constant \(C_X(D)<\infty\), depending
only on \(X\) and \(D\), with the following property.  If
\(\Inj(M)>D+1\), then
\[
 |I|\le C_X(D)\vol(M),
\]
and every scaffold centre has at most \(C_X(D)\) other scaffold centres within
distance \(D\).
\end{lemma}

\begin{proof}
The balls \(B_M(x_i,1/2)\) are pairwise disjoint and embedded, so
\(|I|v(1/2)\le\vol(M)\).  Fix \(x_i\).  If \(d_M(x_i,x_j)<D\), then
\(B_M(x_j,1/2)\subset B_M(x_i,D+1/2)\).  The latter ball is embedded, and the
smaller balls are pairwise disjoint, so the number of such \(x_j\) is at most
\(v(D+1/2)/v(1/2)\).
\end{proof}

\subsection{The sparse good cover}

Take \(t\ge T_X\), where \(T_X\) is large enough for
\cref{lem:log-radii} with \(B_0=4\), and put
\[
 a(t)=\rho(t)+2,
 \qquad
 S(t)=2a(t)=2\rho(t)+4.
\]
Assume throughout this subsection that \(S(t)<R\).  Let \(\calY_t\) be a
Poisson point process on \(M\) of intensity \(\lambda_t=v(t)^{-1}\).  A
scaffold centre \(x_i\) is called \emph{missed} if
\(\calY_t\cap B_M(x_i,\rho(t))=\varnothing\), and \emph{covered} otherwise.
Since \(B_M(x_i,\rho(t))\) is embedded, \cref{lem:log-radii} gives
\begin{equation}\label{eq:missed-prob-thick}
 \bbP(x_i\text{ is missed})
 =\exp\left(-\frac{v(\rho(t))}{v(t)}\right)
 \le t^{-Q_0}.
\end{equation}

Define the random cover
\[
 \calU_t=
 \{B_M(y,a(t)):y\in\calY_t\}
 \cup
 \{B_M(x_i,2):x_i\text{ is missed}\}.
\]

\begin{proposition}\label{prop:good-cover}
For almost every realization of the Poisson process, \(\calU_t\) is a finite
good cover of \(M\).
\end{proposition}

\begin{proof}
The cover is finite almost surely because \(M\) has finite volume.  Let
\(p\in M\), and choose \(x_i\in\calX\) with \(d_M(p,x_i)<1\).  If \(x_i\) is
missed, then \(p\in B_M(x_i,2)\).  If \(x_i\) is covered, choose
\(y\in\calY_t\cap B_M(x_i,\rho(t))\).  Then
\(d_M(p,y)<1+\rho(t)<a(t)\), so \(p\in B_M(y,a(t))\).

Every ball in \(\calU_t\) has radius at most \(a(t)<R\), and is therefore
embedded.  If a finite family of these balls has nonempty intersection, choose
one of them as the first ball.  The sum of its radius and the radius of any
other ball is at most \(S(t)<R\), which is no greater than the injectivity radius at the centre of the reference ball.  The lifting
observation from the preliminaries (\cref{eq:lifting}) identifies the intersection with a convex
intersection of balls in \(X\).  Hence every nonempty finite intersection is
contractible.
\end{proof}

Let \(N_t=N(\calU_t)\), and let \(E_t\) be the number of edges of \(N_t\).
The next estimate is the quantitative heart of the proof.

\begin{proposition}\label{prop:edge-estimate-thick}
There are constants \(C_X,T_X<\infty\), depending only on \(X\), such that,
for every \(t\ge T_X\) with \(S(t)<\Inj(M)\),
\[
 \bbE E_t\le C_X\vol(M)t^{-m/2}(\log t)^2.
\]
\end{proposition}

\begin{proof}
Write \(E_t^{\mathrm{PP}}\), \(E_t^{\mathrm{SS}}\), and
\(E_t^{\mathrm{PS}}\) for the numbers of Poisson--Poisson,
scaffold--scaffold, and mixed edges, respectively.  Thus
\(E_t=E_t^{\mathrm{PP}}+E_t^{\mathrm{SS}}+E_t^{\mathrm{PS}}\).

A Poisson--Poisson edge can occur only when the two Poisson points are at
distance less than \(S(t)\).  By \cref{lem:mecke} and the embeddedness of
radius-\(S(t)\) balls,
\[
 \bbE E_t^{\mathrm{PP}}
 \le \frac{\lambda_t^2}{2}\int_M\vol B_M(y,S(t))\,\dd y
 =\frac12\vol(M)\frac{v(2\rho(t)+4)}{v(t)^2}.
\]
By \eqref{eq:prox-volume}, this is at most
\(C_X\vol(M)t^{-m/2}(\log t)^2\).

A scaffold--scaffold edge can occur only between missed centres at distance
less than four.  Let \(B_t\) be the number of missed scaffold centres.  After
increasing \(T_X\), we may assume \(S(t)>5\).  Then
\cref{lem:scaffold-counts}, applied with \(D=4\), gives
\(E_t^{\mathrm{SS}}\le C_X B_t\).  By
\eqref{eq:missed-prob-thick},
\[
 \bbE B_t\le |I|t^{-Q_0}\le C_X\vol(M)t^{-Q_0},
\]
and hence \(\bbE E_t^{\mathrm{SS}}\le C_X\vol(M)t^{-Q_0}\).

For mixed edges, fix a scaffold centre \(x_i\).  Its radius-two ball can meet a
Poisson ball only if the Poisson centre lies in
\(B_M(x_i,\rho(t)+4)\).  On the event that \(x_i\) is missed, the inner ball
\(B_M(x_i,\rho(t))\) contains no Poisson point. Independence of the Poisson counts in the inner ball and the surrounding annulus gives
\[
 \bbE\left[
  \one_{\{x_i\text{ is missed}\}}
  N_{\calY_t}(B_M(x_i,\rho(t)+4))
 \right]
 \le t^{-Q_0}\frac{v(\rho(t)+4)}{v(t)}
 \le C_X t^{-Q_0}\log t.
\]
Summing over \(i\in I\) yields
\(\bbE E_t^{\mathrm{PS}}\le C_X\vol(M)t^{-Q_0}\log t\).

Since \(Q_0=m/2+4\), the last two contributions are bounded by the
Poisson--Poisson contribution for all sufficiently large \(t\).  Increasing
\(C_X\) and \(T_X\) proves the proposition.
\end{proof}

\begin{proof}[Proof of \cref{thm:main}]
Let \(R=\Inj(M)\) and take \(t=R/4\).  For all sufficiently large \(R\),
depending only on \(X\), one has \(t\ge T_X\) and
\[
 S(t)=2\rho(R/4)+4=R/2+O_X(\log\log R)<R.
\]
Thus \cref{prop:good-cover,prop:edge-estimate-thick} apply.  For every
realization of the Poisson process, \cref{lem:topology-from-nerve} gives
\[
 d(\Gamma)\le E_t,
 \qquad
 \log |H_1(M;\mathbb Z)_{\tors}|\le E_t\log3.
\]
The left-hand sides are deterministic.  Taking expectations and applying
\cref{prop:edge-estimate-thick} gives
\[
 \max\bigl\{d(\Gamma),\log |H_1(M;\mathbb Z)_{\tors}|\bigr\}
 \le C_X\vol(M)t^{-m/2}(\log t)^2.
\]
Since \(t=R/4\) and \(m=r-1\), this is the required estimate after changing
\(C_X\) and increasing \(R_X\).
\end{proof}

\section{First-homology torsion along BS-convergent manifolds}
\label{sec:bs-torsion-manifolds}

In this section, all locally symmetric spaces are manifolds, and all lattices
are torsion-free and irreducible.  Put \(N=\dim X\).  We sample the Poisson
process only in a thick region.  A
scaffold centre is treated separately if it lies too close to the thin part or
if its radius-\(\rho(t)\) ball is missed by the Poisson process.

\subsection{Good covers of homology cores}

\begin{lemma}\label{lem:torsion-from-core-cover}
Let \(M\) and \(C\subset M\) be connected, and let \(C\subset U\subset M\).
Assume that the inclusion \(C\hookrightarrow M\) induces an isomorphism on
\(H_1(-;\mathbb Z)\).  If \(U\) admits a finite good cover \(\calU\), and
\(E(\calU)\) is the number of edges of its nerve, then
\[
        \log |H_1(M;\mathbb Z)_{\tors}|\le E(\calU)\log3.
\]
\end{lemma}

\begin{proof}
Let \(U_0\) be the connected component of \(U\) containing \(C\).  Since every
member of a good cover is connected, the members of \(\calU\) contained in
\(U_0\) form a good cover of \(U_0\), whose nerve has at most
\(E(\calU)\) edges.  We may therefore replace \(U\) by \(U_0\).

Let \(i:C\hookrightarrow U\) and \(j:U\hookrightarrow M\) be the inclusions.
Since \((j\circ i)_*\) is an isomorphism on first homology, the map
\(i_*\circ((j\circ i)_*)^{-1}\) is a section of \(j_*\).  Thus
\(H_1(M;\mathbb Z)\) is a direct summand of \(H_1(U;\mathbb Z)\), and the order
of its torsion subgroup is no larger.  The nerve lemma identifies the latter
group with the first homology of the nerve of \(\calU\).  Applying Gabber's
lemma to its \(2\)-skeleton proves the claim.
\end{proof}

\subsection{The noncompact case}

We use the following form of Gelander's construction.

\begin{lemma}[Gelander's homotopy core]
\label{lem:gelander-noncompact-thick-core}
There is a constant \(\beta_X>0\), depending only on \(X\), such that every
noncompact finite-volume torsion-free irreducible \(X\)-manifold \(M\) contains
a compact connected subset \(C_M\subset M_{\ge\beta_X}\) whose inclusion into
\(M\) is a homotopy equivalence.
\end{lemma}

\begin{proof}
This is the construction in \cite[Theorem~1.5(1)]{GelanderHomotopy}.  The core
may be taken connected because it is homotopy equivalent to the connected
manifold \(M\).
\end{proof}

Fix \(0<\delta<\min\{1,\beta_X/20\}\).  For such a manifold \(M\), choose a
maximal \(\delta\)-separated set
\(\calX=\{x_i:i\in I\}\subset C_M\).  Then the closed balls
\(\overline B_M(x_i,\delta)\) cover \(C_M\), the balls
\(B_M(x_i,\delta/2)\) are pairwise disjoint and embedded, and
\(|I|\le\vol(M)/v(\delta/2)\).

\begin{proposition}\label{prop:bs-torsion-noncompact-estimate}
There are constants \(C_X,b_X,T_X<\infty\), depending only on \(X\), such that
every noncompact finite-volume torsion-free irreducible \(X\)-manifold \(M\)
satisfies, for every \(t\ge T_X\),
\[
        \frac{\log |H_1(M;\mathbb Z)_{\tors}|}{\vol(M)}
        \le C_X\left(
        t^{-m/2}(\log t)^2
        +\log t\,
        \frac{\vol(M_{<3\rho(t)+b_X})}{\vol(M)}
        \right).
\]
\end{proposition}

\begin{proof}
Choose \(T_X\) so that \cref{lem:log-radii} applies with \(B_0=4\delta\).  For \(t\ge T_X\), abbreviate \(\rho=\rho(t)\), and put
\[
        a=\rho+2\delta,
        \qquad
        \omega=2a+1=2\rho+4\delta+1,
        \qquad
        A=\omega+\rho.
\]
Let \(W=M_{\ge\omega}\), and sample on \(W\) a Poisson point process
\(\calY_t\) of intensity \(v(t)^{-1}\).

A scaffold centre \(x_i\) is called \emph{thin} if
\(\inj_{x_i}(M)<A\).  If it is not thin, then
\(B_M(x_i,\rho)\) is embedded and contained in \(W\).  Such a centre is called
\emph{missed} if this ball contains no point of \(\calY_t\), and
\emph{covered} otherwise.  Thin and missed centres are called exceptional.
For every non-thin centre,
\begin{equation}\label{eq:noncompact-miss-prob}
        \bbP(x_i\text{ is missed})
        =\exp\left(-\frac{v(\rho)}{v(t)}\right)
        \le t^{-Q_0}.
\end{equation}

Set
\[
        \calU_t
        =\{B_M(y,a):y\in\calY_t\}
        \cup
        \{B_M(x_i,2\delta):x_i\text{ is exceptional}\},
        \qquad
        U_t=\bigcup_{U\in\calU_t}U.
\]
Then \(C_M\subset U_t\).  Indeed, every point of \(C_M\) is within \(\delta\) of
a scaffold centre.  It is covered by the retained scaffold ball if that centre
is exceptional, and otherwise by a Poisson ball whose centre lies within
\(\rho\) of the scaffold centre.

The family \(\calU_t\) is a finite good cover of \(U_t\).  To see this,
consider a nonempty finite intersection.  If it contains only retained
scaffold balls, choose one of their centres as a reference.  Every other centre
is within \(4\delta<\beta_X\) of it, so the balls lift uniquely to a common
collection of balls in \(X\).  If the intersection contains a Poisson ball
centred at \(y\in W\), use that ball as the reference.  The centre of every
other ball in the intersection is within \(2a<\omega\le
\inj_y(M)\) of \(y\), so the same unique-lifting argument
applies.  In either case, the intersection is isometric to a convex
intersection of balls in \(X\), and is therefore contractible.

Since \(C_M\hookrightarrow M\) is a homotopy equivalence,
\cref{lem:torsion-from-core-cover} gives
\[
        \log |H_1(M;\mathbb Z)_{\tors}|\le E_t\log3,
\]
where \(E_t\) is the number of edges of \(N(\calU_t)\).  Write
\(E_t^{\mathrm{PP}}\), \(E_t^{\mathrm{SS}}\), and
\(E_t^{\mathrm{PS}}\) for the Poisson--Poisson, scaffold--scaffold, and mixed
edge counts.

A Poisson--Poisson edge has endpoints at distance less than
\(2a=2\rho+4\delta\).  Since \(2a<\omega\), the relevant balls centred in \(W\)
are embedded.  By \cref{lem:mecke,eq:prox-volume},
\[
        \bbE E_t^{\mathrm{PP}}
        \le C_X\vol(M)t^{-m/2}(\log t)^2.
\]

Let \(T_t\) be the number of thin scaffold centres and \(B_t\) the number of
missed non-thin centres.  If \(x_i\) is thin, then
\(B_M(x_i,\delta/2)\subset M_{<A+\delta/2}\).  The balls on the left are pairwise
disjoint and embedded, so, after choosing
\(b_X\ge4\delta+1+\delta/2\),
\[
        T_t\le C_X\vol(M_{<3\rho+b_X}).
\]
Moreover, \cref{eq:noncompact-miss-prob} gives
\(\bbE B_t\le C_X\vol(M)t^{-Q_0}\).  The scaffold has uniformly bounded
degree at distance \(4\delta\), and hence
\[
        \bbE E_t^{\mathrm{SS}}
        \le C_X\bigl(\vol(M_{<3\rho+b_X})
        +\vol(M)t^{-Q_0}\bigr).
\]

Finally, a mixed edge incident to \(x_i\) requires a Poisson point in
\(B_M(x_i,\rho+4\delta)\).  For a thin centre, the expected number of such
points is at most
\(v(\rho+4\delta)/v(t)\le C_X\log t\); this upper bound remains valid even when
the quotient ball is not embedded, since its volume is at most the volume of
the corresponding ball in \(X\).  If \(x_i\) is non-thin and missed, the
radius-\(\rho\) ball is contained in \(W\) and contains no Poisson point.
Independence on the surrounding annulus, together with
\cref{eq:noncompact-miss-prob,eq:rho-volume}, gives an expected contribution
at most \(C_X t^{-Q_0}\log t\).  Therefore
\[
        \bbE E_t^{\mathrm{PS}}
        \le C_X\left(
        \log t\,\vol(M_{<3\rho+b_X})
        +\vol(M)t^{-Q_0}\log t
        \right).
\]
Combining the three estimates, using \(\log t\ge1\), and absorbing the terms
involving \(t^{-Q_0}\) into the Poisson--Poisson term proves the proposition.
\end{proof}

\subsection{The compact case}

For a compact manifold \(M\), write
\(\varepsilon_x=\min\{1,\inj_x(M)\}\).

\begin{theorem}[Fr\k{a}czyk--Hurtado--Raimbault]
\label{thm:fhr-variable-radius-cover}
There are constants \(C_0,\Delta,b_0<\infty\), depending only on \(X\), such
that every compact torsion-free irreducible \(X\)-manifold \(M\) admits a
finite set \(S\subset M\) with the following properties.  For \(s\in S\), put
\(r_s=\varepsilon_s/6\).  Then
\(\mathcal V=\{B_M(s,r_s):s\in S\}\) is a finite good cover of \(M\), the
nerve of \(\mathcal V\) has vertex degree at most \(\Delta\), and
\(|S|\le C_0\vol(M)\).  Moreover, for every \(A\ge1\),
\[
        \#\{s\in S:\inj_s(M)<A\}
        \le C_0\int_{M_{<A+b_0}}\varepsilon_x^{-N}\,\dd x.
\]
\end{theorem}

\begin{proof}
The construction of \(S\), the good-cover property, and the bounded-degree
property are contained in
\cite[Proposition~6.1 and Lemma~6.3]{FraczykHurtadoRaimbaultTopComplexity}; the
linear bound on \(|S|\) is \cite[Theorem~A]{FraczykHurtadoRaimbaultTopComplexity}.

For the localized estimate, use the functions from the proof of
\cite[Lemma~6.3]{FraczykHurtadoRaimbaultTopComplexity}.  For \(s\in S\), set
\[
 F_s(x)=
 \begin{cases}
  \vol(B_M(s,\varepsilon_s/2))^{-1},
       &x\in B_M(s,\varepsilon_s/4),\\
  0,&\text{otherwise}.
 \end{cases}
\]
That proof shows that \(\int_M F_s\,\dd x\asymp_X 1\) and
\(\sum_{s\in S}F_s(x)\le C_X\varepsilon_x^{-N}\).  If
\(\inj_s(M)<A\) and \(x\in\operatorname{supp}F_s\), then the
\(1\)-Lipschitz property of the injectivity radius gives
\(\inj_x(M)<A+1\).  Summing over these centres and integrating
proves the last estimate, after increasing the constants.
\end{proof}

\begin{proposition}\label{prop:bs-torsion-compact-estimate}
There are constants \(C_X,b_X,T_X<\infty\), depending only on \(X\), such that
every compact torsion-free irreducible \(X\)-manifold \(M\) satisfies, for
every \(t\ge T_X\),
\[
        \frac{\log |H_1(M;\mathbb Z)_{\tors}|}{\vol(M)}
        \le C_X\left(
        t^{-m/2}(\log t)^2
        +\frac{\log t}{\vol(M)}
        \int_{M_{<3\rho(t)+b_X}}\varepsilon_x^{-N}\,\dd x
        \right).
\]
\end{proposition}

\begin{proof}
Let \(S\) and \(\mathcal V\) be as in
\cref{thm:fhr-variable-radius-cover}.  Choose \(T_X\) so that
\cref{lem:log-radii} applies with \(B_0=2\).  For
\(t\ge T_X\), abbreviate \(\rho=\rho(t)\), and put
\[
        a=\rho+1,
        \qquad
        \omega=2a+1=2\rho+3,
        \qquad
        A=\omega+\rho.
\]
Let \(W=M_{\ge\omega}\), and sample on \(W\) a Poisson point process
\(\calY_t\) of intensity \(v(t)^{-1}\).

Call \(s\in S\) \emph{thin} if \(\inj_s(M)<A\).  If \(s\) is
not thin, then \(B_M(s,\rho)\) is embedded and contained in \(W\).  Such a
centre is called \emph{missed} if this ball contains no point of \(\calY_t\),
and \emph{covered} otherwise.  For every non-thin centre,
\begin{equation}\label{eq:compact-miss-prob}
        \bbP(s\text{ is missed})
        =\exp\left(-\frac{v(\rho)}{v(t)}\right)
        \le t^{-Q_0}.
\end{equation}
Define
\[
        \calU_t
        =\{B_M(y,a):y\in\calY_t\}
        \cup
        \{B_M(s,r_s):s\in S\text{ is thin or missed}\}.
\]

The family \(\calU_t\) covers \(M\).  Indeed, if
\(p\in B_M(s,r_s)\) for a ball in the original cover \(\mathcal V\), then
either this ball is retained, or \(s\) is covered and there is a Poisson point
\(y\in B_M(s,\rho)\).  In the latter case,
\(d_M(p,y)<r_s+\rho<a\).

The cover is good.  An intersection involving only retained balls is an
intersection from the good cover \(\mathcal V\).  If an intersection contains
a Poisson ball centred at \(y\in W\), use that ball as the reference.  Every
other centre in the intersection is at distance less than \(2a<\omega\) from
\(y\), so all balls lift uniquely to a common collection of convex balls in
\(X\).  The intersection is therefore contractible.  By
\cref{lem:topology-from-nerve},
\[
        \log |H_1(M;\mathbb Z)_{\tors}|\le E_t\log3,
\]
where \(E_t\) is the number of edges of \(N(\calU_t)\).

As before, write \(E_t^{\mathrm{PP}}\), \(E_t^{\mathrm{SS}}\), and
\(E_t^{\mathrm{PS}}\) for the three edge types.  Since a Poisson--Poisson edge
has endpoints at distance less than \(2a=2\rho+2\),
\cref{lem:mecke,eq:prox-volume} give
\[
        \bbE E_t^{\mathrm{PP}}
        \le C_X\vol(M)t^{-m/2}(\log t)^2.
\]

By \cref{thm:fhr-variable-radius-cover}, after increasing \(b_X\), the number
of thin centres is at most
\[
        C_X\int_{M_{<3\rho+b_X}}\varepsilon_x^{-N}\,\dd x.
\]
The expected number of missed non-thin centres is at most
\(C_X\vol(M)t^{-Q_0}\), by \cref{eq:compact-miss-prob} and the bound
\(|S|\le C_X\vol(M)\).  Since the exceptional--exceptional edges form a
subgraph of the bounded-degree nerve of \(\mathcal V\),
\[
        \bbE E_t^{\mathrm{SS}}
        \le C_X\left(
        \int_{M_{<3\rho+b_X}}\varepsilon_x^{-N}\,\dd x
        +\vol(M)t^{-Q_0}
        \right).
\]

A mixed edge incident to \(s\) requires a Poisson point in
\(B_M(s,\rho+2)\).  A thin centre contributes at most \(C_X\log t\) in
expectation, by \cref{eq:rho-volume}.  If \(s\) is non-thin and missed,
independence between its radius-\(\rho\) ball and the surrounding annulus gives
an expected contribution at most \(C_X t^{-Q_0}\log t\).  Consequently,
\[
        \bbE E_t^{\mathrm{PS}}
        \le C_X\left(
        \log t\int_{M_{<3\rho+b_X}}\varepsilon_x^{-N}\,\dd x
        +\vol(M)t^{-Q_0}\log t
        \right).
\]
Combining the estimates and absorbing the terms involving \(t^{-Q_0}\) into
the Poisson--Poisson term proves the proposition.
\end{proof}

\subsection{The compact weighted thin part}

\begin{lemma}\label{lem:compact-weighted-thin-bs-log}
Let \(M_n=\Gamma_n\backslash X\) be a compact torsion-free irreducible
BS-convergent sequence.  Every subsequence has a further subsequence, still
denoted \(M_n\), and numbers \(t_n\to\infty\) such that
\[
        \frac{\log t_n}{\vol(M_n)}
        \int_{(M_n)_{<3\rho(t_n)+b_X}}
             \varepsilon_{n,x}^{-N}\,\dd x
        \longrightarrow0,
\]
where \(b_X\) is the constant from
\cref{prop:bs-torsion-compact-estimate} and
\(\varepsilon_{n,x}=\min\{1,\inj_x(M_n)\}\).
\end{lemma}

\begin{proof}
By Margulis arithmeticity, the lattices \(\Gamma_n\) are arithmetic.  Let
\(k_n\) be the adjoint trace field of \(\Gamma_n\), and put
\(d_n=[k_n:\mathbb Q]\).  After passing to a subsequence, either \(d_n\) is
bounded or \(d_n\to\infty\).

Suppose first that \(d_n\) is bounded.  The Dobrowolski injectivity-radius
estimate used in
\cite[Proposition~2.3]{FraczykHurtadoRaimbaultTopComplexity} gives a uniform
lower bound \(\Inj(M_n)\ge\iota>0\).  Hence
\(\varepsilon_{n,x}^{-N}\le C_X\).  Let \(j_0\) be an integer at least as
large as the threshold in \cref{prop:bs-torsion-compact-estimate}.  For each
integer \(j\ge j_0\), BS convergence gives
\[
        \frac{\vol((M_n)_{<3\rho(j)+b_X})}{\vol(M_n)}
        \longrightarrow0.
\]
Choose a strictly increasing sequence of indices \(N_j\), \(j\ge j_0\),
such that, for \(n\ge N_j\), this ratio is at most
\(1/(j\log(j+2))\).  Set \(t_n=j\) when \(N_j\le n<N_{j+1}\), and define
\(t_n\) arbitrarily for the finitely many earlier indices.  Then
\(t_n\to\infty\), and the expression in the
statement is at most
\(C_X\log t_n/(t_n\log(t_n+2))\), which tends to zero.

Suppose now that \(d_n\to\infty\).  By
\cite[Theorem~D]{FraczykHurtadoRaimbaultTopComplexity}, there are constants
\(c,\eta_1>0\), depending only on \(X\), such that
\[
        \vol((M_n)_{<\eta_1d_n})
        \le e^{-cd_n}\vol(M_n)
\]
for all sufficiently large \(n\).  The same Dobrowolski estimate gives
\(\Inj(M_n)\ge \kappa_X/(\log(2d_n))^3\).  Put
\(t_n=\lfloor\eta_1d_n/10\rfloor\).  Since
\(\rho(t)=t+O_X(\log\log t)\), one has
\(3\rho(t_n)+b_X\le\eta_1d_n\) for all large \(n\).  Therefore
\[
\begin{aligned}
        &\frac{\log t_n}{\vol(M_n)}
        \int_{(M_n)_{<3\rho(t_n)+b_X}}
             \varepsilon_{n,x}^{-N}\,\dd x \\
        &\qquad\le
        C_X\log t_n\,(\log(2d_n))^{3N}e^{-cd_n}
        \longrightarrow0.
\end{aligned}
\]
\end{proof}

\begin{proof}[Proof of \cref{thm:intro-bs-h1-torsion}]
It is enough to prove that every subsequence has a further subsequence along
which the normalized torsion tends to zero.  After passing to a subsequence,
assume that all \(M_n\) are noncompact or that all are compact.

In the noncompact case, let \(b_X\) be as in
\cref{prop:bs-torsion-noncompact-estimate}.  A diagonal argument using BS
convergence gives numbers \(t_n\to\infty\), all above the threshold in
\cref{prop:bs-torsion-noncompact-estimate}, such that
\[
        \log t_n\,
        \frac{\vol((M_n)_{<3\rho(t_n)+b_X})}{\vol(M_n)}
        \longrightarrow0.
\]
For example, one may choose \(t_n=j\) on successive ranges of indices on which
the displayed thin-part ratio is at most \(1/(j\log(j+2))\).  Applying
\cref{prop:bs-torsion-noncompact-estimate} with these \(t_n\) gives the desired
conclusion, since \(t_n^{-m/2}(\log t_n)^2\to0\).

In the compact case, let \(b_X\) be as in
\cref{prop:bs-torsion-compact-estimate}.  By
\cref{lem:compact-weighted-thin-bs-log}, after passing to a further
subsequence there are \(t_n\to\infty\) such that
\[
        \frac{\log t_n}{\vol(M_n)}
        \int_{(M_n)_{<3\rho(t_n)+b_X}}
             \varepsilon_{n,x}^{-N}\,\dd x
        \longrightarrow0.
\]
Applying \cref{prop:bs-torsion-compact-estimate} with these \(t_n\) proves the
compact case and completes the proof.
\end{proof}

\section{Core scaffolds for general lattices}\label{sec:core-scaffolds}

We now prepare the scaffold used in the proof of the rank theorem for general
BS-convergent orbifolds.  Let \(O=\Gamma\backslash X\) be the finite-volume
orbifold associated with an irreducible lattice \(\Gamma<G\), and let
\(\pi:X\to O\) be the quotient map.  For \(x\in O\), choose a lift
\(\widetilde x\in X\) and define
\[
        \inj_x(O)
        =\frac12\inf_{\gamma\in\Gamma\setminus\{1\}}
        d_X(\widetilde x,\gamma\widetilde x).
\]
This is independent of the lift and is zero at orbifold singular points.  For
\(L>0\), put
\(O_{<L}=\{x:\inj_x(O)<L\}\) and
\(O_{\ge L}=\{x:\inj_x(O)\ge L\}\).  The set \(O_{>0}\) is the
free part of the orbifold.

\begin{lemma}\label{lem:orbifold-inj-lipschitz}
The function \(x\mapsto\inj_x(O)\) is \(1\)-Lipschitz.
\end{lemma}

\begin{proof}
For each fixed \(\gamma\in\Gamma\), the displacement function
\(u\mapsto d_X(u,\gamma u)\) is \(2\)-Lipschitz.  Hence
\(u\mapsto\frac12d_X(u,\gamma u)\) is \(1\)-Lipschitz, and so is the infimum
over \(\gamma\ne1\).  This function is \(\Gamma\)-invariant and therefore
descends to \(O\).
\end{proof}

\begin{theorem}[Gelander core]\label{thm:gelander-core-orbifold}
There is a constant \(\alpha_X>0\), depending only on \(X\), such that every
irreducible lattice \(\Gamma<G\) admits a compact subset
\(\mathfrak C_O\subset O=\Gamma\backslash X\) with the following properties:
\begin{enumerate}[label=(\roman*),leftmargin=2.4em]
    \item \(\mathfrak C_O\subset O_{\ge\alpha_X}\);
    \item its full inverse image
    \(\widetilde{\mathfrak C}_O=\pi^{-1}(\mathfrak C_O)\) is nonempty and
    connected;
    \item the restriction
    \(\widetilde{\mathfrak C}_O\to\mathfrak C_O\) is a regular covering with
    deck group \(\Gamma\).  In particular,
    \(\pi_1(\mathfrak C_O)\twoheadrightarrow\Gamma\).
\end{enumerate}
\end{theorem}

\begin{proof}
Gelander constructs a \(\Gamma\)-invariant singular set
\(\widetilde N\subset X\) and a \(\Gamma\)-invariant Morse function
\(\widetilde\psi\) on \(X\setminus\widetilde N\); see
\cite[Sections~2.3--2.6]{GelanderRank}.  In the irreducible case, the lifted
sublevel set \(\widetilde\psi_{\le0}\) is nonempty and connected, its quotient
is compact, and every point of the quotient has injectivity radius at least
\(\epsilon/(2\mu)\), for constants depending only on \(G\); see
\cite[Theorem~2.9, Corollaries~2.10--2.11 and 2.13, and
Remark~2.12]{GelanderRank}.  Taking
\(\mathfrak C_O=\psi_{\le0}\) and \(\alpha_X=\epsilon/(2\mu)\) gives the
first two assertions.  Since \(\mathfrak C_O\) lies in the free part, the
restriction of the quotient map is a regular covering with deck group
\(\Gamma\), which gives the last assertion.
\end{proof}

Fix once and for all \(0<\eta<\alpha_X/10\).  Choose a maximal
\(\eta\)-separated set
\(\calX=\{x_i:i\in I\}\subset\mathfrak C_O\).  The closed balls
\(\overline B_O(x_i,\eta)\) cover \(\mathfrak C_O\).  Define a finite graph
\(H\) with vertex set \(\calX\) by joining distinct vertices whose distance in
\(O\) is at most \(2\eta\).  Each edge is represented by its unique minimizing
geodesic segment.  Let \(\widehat H\) be the full lift of this graph to \(X\).
Equivalently, its vertices are the full lifts of the points of \(\calX\), and
two distinct lifted vertices are joined when their distance in \(X\) is at
most \(2\eta\).

\begin{lemma}[Core scaffold]\label{lem:core-graph}
The graph \(\widehat H\) is connected.  There is a constant
\(\Delta_X<\infty\), depending only on \(X\), such that every vertex of
\(\widehat H\) has degree at most \(\Delta_X\).  Consequently,
\[
        |E(H)|\le \Delta_X|I|,
        \qquad
        |I|\le \frac{\vol(O)}{v(\eta/2)}.
\]
Moreover, every geodesic segment representing an edge of \(H\) lies in the
free part of \(O\).
\end{lemma}

\begin{proof}
The full lifted set of scaffold centres is \(\eta\)-separated in \(X\).  For a
fixed lifted centre, the balls of radius \(\eta/2\) around its neighbours are
pairwise disjoint and lie in the ball of radius \(5\eta/2\) around that centre.
This gives a degree bound depending only on \(X\) and \(\eta\), hence only on
\(X\).  Downstairs, the balls \(B_O(x_i,\eta/2)\) are pairwise disjoint and
embedded, so \(|I|v(\eta/2)\le\vol(O)\).  The asserted edge bound follows from
the degree bound.

The lifted open \(\eta\)-balls centred at the vertices of \(\widehat H\)
cover \(\widetilde{\mathfrak C}_O\).  Their intersections with this connected set
form an open cover with connected intersection graph.  Hence the intersection
graph of the balls, and therefore \(\widehat H\), is connected.  Finally, if
\(z\) lies on an edge segment, then it is within \(2\eta\) of an
endpoint in \(O_{\ge\alpha_X}\).  By \cref{lem:orbifold-inj-lipschitz},
\(\inj_z(O)\ge\alpha_X-2\eta>0\), so the segment lies in the
free part.
\end{proof}

\section{Sublinear rank for BS-convergent orbifolds}\label{sec:bs-rank}

\begin{proof}[Proof of \cref{thm:intro-bs-orbifold}]
Choose \(T_X\) so that \cref{lem:log-radii} applies with \(B_0=2\eta\).  For \(t\ge T_X\), put
\[
        \sigma(t)=2\rho(t)+2\eta,
        \qquad
        \omega(t)=\sigma(t)+\eta,
        \qquad
        A(t)=\omega(t)+\rho(t),
        \qquad
        D(t)=A(t)+\eta=3\rho(t)+4\eta.
\]

We first choose a scale along the sequence.  Let \(j_0=\lceil T_X\rceil\).  For
each integer \(j\ge j_0\), BS convergence gives
\(\vol((O_n)_{<D(j)})/\vol(O_n)\to0\).  Choose a strictly increasing sequence
of integers \(N_j\) such that, whenever \(n\ge N_j\),
\[
        \frac{\vol((O_n)_{<D(j)})}{\vol(O_n)}\le\frac1j.
\]
For \(n\ge N_{j_0}\), define
\(t_n=\max\{j\ge j_0:N_j\le n\}\), and define \(t_n\) arbitrarily for the
finitely many remaining indices.  Then \(t_n\to\infty\) and
\[
        \frac{\vol((O_n)_{<D(t_n)})}{\vol(O_n)}
        \le\frac1{t_n}\longrightarrow0.
\]

Fix a sufficiently large \(n\), and abbreviate
\(O=O_n\), \(\Gamma=\Gamma_n\), and \(t=t_n\).  Let
\(\mathfrak C_O\), \(H\), and \(\widehat H\) be as in
\cref{sec:core-scaffolds}.  Put \(W=O_{\ge\omega(t)}\), and sample on \(W\) a
Poisson point process \(\calY\) of intensity \(v(t)^{-1}\).

A scaffold vertex \(x_i\) is called \emph{thin} if
\(\inj_{x_i}(O)<A(t)\).  If it is not thin, then
\(B_O(x_i,\rho(t))\) is embedded and contained in \(W\).  Such a vertex is
called \emph{missed} if this ball contains no point of \(\calY\), and
\emph{covered} otherwise.  Thin and missed vertices are called
\emph{exceptional}.  For every non-thin vertex,
\begin{equation}\label{eq:orbifold-miss-prob}
        \bbP(x_i\text{ is missed})
        =\exp\left(-\frac{v(\rho(t))}{v(t)}\right)
        \le t^{-Q_0}.
\end{equation}

We define a finite random graph \(\mathcal G_t\) in \(O\).  Its Poisson
vertices are the points \(y\in\calY\) that lie in
\(B_O(x_i,\rho(t))\) for at least one covered scaffold vertex \(x_i\).  Its
scaffold vertices are the exceptional vertices and the covered vertices that
are adjacent in \(H\) to an exceptional vertex.  The graph has the following
edges:
\begin{enumerate}[label=(\roman*),leftmargin=2.4em]
    \item an edge between distinct Poisson vertices \(y,z\) whenever
    \(d_O(y,z)\le\sigma(t)\);
    \item every scaffold edge of \(H\) with at least one exceptional endpoint;
    \item for each covered scaffold vertex retained in the graph, one edge to
    a chosen Poisson point in its radius-\(\rho(t)\) ball.
\end{enumerate}
The choices in (iii) are made in the quotient and then lifted equivariantly.
The estimates below do not depend on these choices.  Edges of types (i) and
(iii) are represented by their unique minimizing geodesic segments.  Indeed,
for (i) either endpoint lies in \(O_{\ge\omega(t)}\) and
\(\sigma(t)<\omega(t)\), while for (iii) the scaffold vertex has injectivity
radius at least \(A(t)>\rho(t)\).  Let \(\widehat{\mathcal G}_t\) be the full
lift of this graph to \(X\).

We claim that \(\widehat{\mathcal G}_t\) is connected.  For this connectivity argument, classify every lifted scaffold vertex
according to its projection, whether or not that vertex is retained in
\(\widehat{\mathcal G}_t\).  If there are no covered
vertices, every edge of the connected graph \(\widehat H\) is retained, and
the claim is immediate.  Otherwise, for each covered lifted vertex
\(\widetilde x\), let
\[
        P_{\widetilde x}
        =\widetilde{\calY}\cap B_X(\widetilde x,\rho(t)),
\]
where \(\widetilde{\calY}\) is the full lift of \(\calY\).  This set is
nonempty, and any two of its points are at distance at most
\(2\rho(t)<\sigma(t)\).  If \(\widetilde x\) and \(\widetilde x'\) are
adjacent covered vertices, \(y\in P_{\widetilde x}\), and
\(z\in P_{\widetilde x'}\), then
\(d_X(y,z)\le2\rho(t)+2\eta=\sigma(t)\).  Hence all Poisson vertices associated
to a connected component of the covered subgraph of \(\widehat H\) lie in a
single component of \(\widehat{\mathcal G}_t\).

Each connected component of the exceptional subgraph of \(\widehat H\) is
retained by edges of type (ii).  Since \(\widehat H\) is connected, the
incidence graph of the covered and exceptional components is connected.  At
every edge between an exceptional component and a covered component, the
covered endpoint is retained, its scaffold edge to the exceptional component
is present, and an edge of type (iii) joins it to the corresponding Poisson
component.  Thus every adjacency in the incidence graph is realized by a path
in \(\widehat{\mathcal G}_t\), proving connectivity.

The action of \(\Gamma\) on the geometric realization of
\(\widehat{\mathcal G}_t\) is free.  Poisson vertices lie over
\(O_{\ge\omega(t)}\), and scaffold vertices lie in the free part.  Along an
edge of type (i), the injectivity radius is at least
\(\omega(t)-\sigma(t)=\eta>0\).  Along an edge of type (iii), it is at least
\(A(t)-\rho(t)=\omega(t)>0\).  Edges of type (ii) lie in the free part by
\cref{lem:core-graph}.  The quotient graph is finite almost surely, so
\cref{lem:graph-surjection} gives
\[
        d(\Gamma)\le |E(\mathcal G_t)|
\]
for almost every realization.

It remains to estimate the expected number of quotient edges.  The number of
edges of type (i) is bounded by the number of pairs of points of \(\calY\) at
distance at most \(\sigma(t)\).  Since a radius-\(\sigma(t)\) ball centred in
\(W\) is embedded, \cref{lem:mecke,eq:prox-volume} give
\[
\begin{aligned}
        \bbE\#\{\text{edges of type (i)}\}
        &\le \frac{1}{2v(t)^2}
        \int_W\vol(B_O(x,\sigma(t))\cap W)\,\dd x \\
        &\le C_X\vol(O)t^{-m/2}(\log t)^2.
\end{aligned}
\]

Let \(T\) be the set of thin scaffold vertices.  If \(x_i\in T\), then
\(B_O(x_i,\eta/2)\subset O_{<D(t)}\), by
\cref{lem:orbifold-inj-lipschitz}.  These balls are pairwise disjoint and
embedded, so
\[
        |T|\le\frac{\vol(O_{<D(t)})}{v(\eta/2)}.
\]
Because the scaffold graph has uniformly bounded degree, the number of edges
of types (ii) and (iii) arising from thin vertices is at most
\(C_X\vol(O_{<D(t)})\).

Let \(B\) be the number of missed non-thin vertices.  By
\cref{eq:orbifold-miss-prob} and \cref{lem:core-graph},
\(\bbE B\le C_X\vol(O)t^{-Q_0}\).  The expected number of edges of types (ii)
and (iii) arising from missed vertices is therefore at most
\(C_X\vol(O)t^{-Q_0}\).  Combining these estimates gives
\[
        \bbE|E(\mathcal G_t)|
        \le C_X\left(
        \vol(O)t^{-m/2}(\log t)^2
        +\vol(O)t^{-Q_0}
        +\vol(O_{<D(t)})
        \right).
\]
Applying this with \(O=O_n\) and \(t=t_n\), and dividing by \(\vol(O_n)\),
yields
\[
        \frac{d(\Gamma_n)}{\vol(O_n)}
        \le C_X\left(
        t_n^{-m/2}(\log t_n)^2+t_n^{-Q_0}
        +\frac{\vol((O_n)_{<D(t_n)})}{\vol(O_n)}
        \right).
\]
The right-hand side tends to zero, which proves the theorem.
\end{proof}

\begin{proof}[Proof of \cref{cor:intro-modp}]
For every field \(\mathbb F\), the images in
\(H_1(\Gamma_n;\mathbb F)\) of a generating set of \(\Gamma_n\) span that
vector space.  Hence
\(\dim_{\mathbb F}H_1(\Gamma_n;\mathbb F)\le d(\Gamma_n)\), uniformly in
\(\mathbb F\).  The conclusion follows from
\cref{thm:intro-bs-orbifold}.
\end{proof}
\section{Affine buildings}\label{sec:affine-buildings}

In this section we prove \cref{intro:buildings}. Since the argument is essentially the same, we highlight the main changes. Let
\(\mathcal B\) be a thick, locally finite, affine building of
Euclidean rank \(r\ge2\), equipped with its standard piecewise-Euclidean
CAT\((0)\) metric and with \(r\)-dimensional Hausdorff measure \(\mu\).
Put \(m=r-1\).  For \(x\in\mathcal B\) and \(s>0\), write
\[
 v_x(s)=\mu(B_{\mathcal B}(x,s)),
 \qquad
 v_-(s)=\inf_{x\in\mathcal B}v_x(s),
 \qquad
 v_+(s)=\sup_{x\in\mathcal B}v_x(s).
\]
Since we only consider compact quotients of \(\mathcal B\), we may assume
that \(\mathcal B\) admits a cocompact lattice.  In particular, \(\mathcal B\) has bounded local geometry.

The argument is the same scaffolded Poisson construction as before.  There
are three points that require modification.  First, the building is not assumed
homogeneous, so lower and upper ball-volume functions replace the single
function \(v(s)\).  Second, convexity of balls in the Hadamard manifold \(X\)
is replaced by convexity of balls in the CAT\((0)\) space \(\mathcal B\).
Third, the fixed polyhedral structure gives a uniform radius for the scaffold
for all torsion-free quotients. Thus, no variable-radius cover is
needed in the thin part.

\begin{lemma}
\label{lem:building-geometric-input}
Assume that \(\mathcal B\) admits a compact quotient.  There are constants
\(
 \varepsilon_{\mathcal B},a_{\mathcal B},
 A_{\mathcal B},h_{\mathcal B}>0
\)
depending only on \(\mathcal B\) such that the following hold.
\begin{enumerate}[label=(\roman*),leftmargin=2.4em]
    \item For every \(s\ge1\),
    \[
     a_{\mathcal B}e^{h_{\mathcal B}s}s^{(r-1)/2}
     \le v_-(s)
     \le v_+(s)
     \le
     A_{\mathcal B}e^{h_{\mathcal B}s}s^{(r-1)/2}.
    \]

    \item If \(\Gamma<\operatorname{Aut}(\mathcal B)\) is a torsion-free
    discrete subgroup and \(Y=\Gamma\backslash\mathcal B\), then
    \[
     \Inj(Y)\ge\varepsilon_{\mathcal B}.
    \]
    In particular, every ball of radius
    \(\varepsilon_{\mathcal B}\) in \(Y\) lifts isometrically to
    \(\mathcal B\).
\end{enumerate}
\end{lemma}

\begin{proof}
Write
\[
 \mathcal B=\mathcal B_1\times\cdots\times\mathcal B_k
\]
for the product decomposition into irreducible affine factors, and put
\(r_j=\rank\mathcal B_j\), so that
\(\sum_{j=1}^k r_j=r\).  The standard metric is the
\(\ell^2\)-product metric.

After passing to a finite-index subgroup of a fixed cocompact lattice, we
may assume that the factors are preserved.  The projected action on each
factor has finitely many vertex orbits.  In particular, every rank-one
factor is a quasi-transitive locally finite thick tree.

Let \(\mu_j\) denote the measure on \(\mathcal B_j\).  For every
factor there are constants \(h_j,c_j,C_j>0\) such that, uniformly in the
centre \(x_j\) and for every integer \(n\ge1\),
\begin{equation}\label{eq:factor-shell-growth}
 c_j e^{h_jn}(n+1)^{(r_j-1)/2}
 \le
 \mu_j\bigl(B_{\mathcal B_j}(x_j,n+1)
             \setminus B_{\mathcal B_j}(x_j,n)\bigr)
 \le
 C_j e^{h_jn}(n+1)^{(r_j-1)/2}.
\end{equation}
Indeed, if \(r_j\ge2\), then \(\mathcal B_j\) is regular by
\cite[Theorem~2.4]{ParkinsonBuildingsHecke}, and
\cite[Theorem~5.15]{ParkinsonBuildingsHecke} gives
\eqref{eq:factor-shell-growth}.  If \(r_j=1\), then
\(\mathcal B_j\) is a quasi-transitive thick tree.  Looking at the transition matrix on the finitely many
orbits of oriented edges, one gets
\[
 \mu_j\bigl(B_{\mathcal B_j}(x_j,n+1)
             \setminus B_{\mathcal B_j}(x_j,n)\bigr)
 \asymp_{\mathcal B_j}e^{h_jn},
\]
which is precisely \eqref{eq:factor-shell-growth} in rank one.

Put
\[
 h_{\mathcal B}=\left(\sum_{j=1}^k h_j^2\right)^{1/2},
 \qquad
 \alpha_j=\frac{r_j-1}{2}.
\]
After decomposing a product ball into its factors, we get
\[
 v_x(s)
 \asymp_{\mathcal B}
 e^{h_{\mathcal B}s}
 s^{\sum_j\alpha_j+(k-1)/2}
\]
This proves (i), uniformly in the centre.

For (ii), apply the discreteness result for translation lengths of
\cite{BridsonSemisimplicity}.  It gives a constant
\(2\varepsilon_{\mathcal B}>0\) such that every cellular automorphism of
\(\mathcal B\) with positive translation length has translation length at
least \(2\varepsilon_{\mathcal B}\).  If \(\Gamma\) is discrete and
torsion-free, then no nontrivial element of \(\Gamma\) is elliptic: a point
stabilizer in \(\operatorname{Aut}(\mathcal B)\) is compact, so its
intersection with \(\Gamma\) is finite.  Thus every
\(\gamma\in\Gamma\setminus\{1\}\) has translation length at least
\(2\varepsilon_{\mathcal B}\), proving the claim.
\end{proof}

Choose \(c_{\mathcal B}\) sufficiently large and set
\[
 \rho_{\mathcal B}(t)
 =
 t+\frac{1}{h_{\mathcal B}}\log\log t+c_{\mathcal B}
 \qquad (t\ge e).
\]
By \cref{lem:building-geometric-input}, the analogues of
\cref{eq:miss-prob,eq:rho-volume,eq:prox-volume} hold with
\(v_-\) in the lower-volume estimate and \(v_+\) in the upper-volume
estimates.  In particular, after increasing \(c_{\mathcal B}\),
\[
 \exp\left(
   -\frac{v_-(\rho_{\mathcal B}(t))}{v_+(t)}
 \right)
\]
is bounded by an arbitrarily prescribed negative power of \(t\), while
\[
 \frac{v_+(2\rho_{\mathcal B}(t)+O_{\mathcal B}(1))}
      {v_+(t)^2}
 \ll_{\mathcal B}
 t^{(1-r)/2}(\log t)^2.
\]

\begin{proof}[Proof of \cref{intro:buildings}]
Fix
\(
 0<\varepsilon<\frac{1}{10}\varepsilon_{\mathcal B}.
\)
For a compact quotient \(Y=\Gamma\backslash\mathcal B\), choose a maximal
\(\varepsilon\)-separated scaffold.  Its
\(\varepsilon\)-balls cover \(Y\), its
\(\varepsilon/2\)-balls are pairwise disjoint and embedded, and its local
degree at every fixed scale is bounded in terms of \(\mathcal B\) alone.

We first consider the large-injectivity-radius case.  Sample a Poisson point
process of intensity \(v_+(t)^{-1}\), and replace a scaffold ball whenever
there is a Poisson point in its radius-\(\rho_{\mathcal B}(t)\)
neighbourhood.  This is exactly the construction of
\cref{sec:thick-case}, with \(v_-\) used to bound the probability of a
missed scaffold centre and \(v_+\) used in the pair estimates.

When all relevant radii are smaller than \(\Inj(Y)\), every nonempty
intersection in the resulting cover lifts to an intersection of metric
balls in \(\mathcal B\).  Metric balls in a CAT\((0)\) space are convex, so
the cover is good.  The Mecke formula and the preceding volume estimates
give
\[
 \bbE E_t
 \le
 C_{\mathcal B}\mu(Y)
 t^{(1-r)/2}(\log t)^2,
\]
where \(E_t\) is the number of edges of the nerve.  Since
\(\mathcal B\) is contractible and the torsion-free action is free,
\(Y\) is a \(K(\Gamma,1)\).  The nerve lemma and the same graph and Gabber
estimates used earlier therefore give
\[
 d(\Gamma)\le E_t,
 \qquad
 \log|H_1(\Gamma;\mathbb Z)_{\tors}|\le E_t\log3.
\]
Taking \(t=R/4\), where \(R=\Inj(Y)\), gives the corresponding
quantitative large-injectivity-radius estimate.

For the BS-convergent statement, sample the Poisson process only in a thick
region.  A scaffold centre is declared exceptional if it is too close to
the thin part or if its radius-\(\rho_{\mathcal B}(t)\) neighbourhood is
missed.  Retain a fixed small scaffold ball at every exceptional centre and
use the large Poisson balls elsewhere.

The uniform constant \(\varepsilon_{\mathcal B}\) is the only new
small-scale input needed here.  It ensures that intersections consisting
only of retained scaffold balls lift, while intersections containing a
Poisson ball lift because the Poisson centre was chosen in the thick
region.  Thus the resulting cover is again good.

The edge decomposition is the same as in the manifold case.  Poisson--Poisson
edges contribute
\[
 C_{\mathcal B}\mu(Y)t^{(1-r)/2}(\log t)^2.
\]
Missed non-thin centres have polynomially small expected density.  The
disjoint \(\varepsilon/2\)-balls around thin scaffold centres show that
their number is bounded by a constant times the volume of a slightly larger
thin part.  Mixed edges incident to such centres introduce one factor of
\(\log t\).  Consequently, there are constants
\(C_{\mathcal B},b_{\mathcal B},T_{\mathcal B}<\infty\) such that
\begin{equation}\label{eq:building-master-estimate}
 \frac{1}{\mu(Y)}
 \max\left\{
   d(\Gamma),
   \log|H_1(\Gamma;\mathbb Z)_{\tors}|
 \right\}
 \le
 C_{\mathcal B}\left(
   t^{(1-r)/2}(\log t)^2
   +
   \log t\,
   \frac{
     \mu\bigl(Y_{<3\rho_{\mathcal B}(t)+b_{\mathcal B}}\bigr)
   }{\mu(Y)}
 \right)
\end{equation}
for every \(t\ge T_{\mathcal B}\).

Now let \(Y_n=\Gamma_n\backslash\mathcal B\) BS-converge to
\(\mathcal B\).  A diagonal argument provides \(t_n\to\infty\) such that
\[
 \log t_n\,
 \frac{
   \mu\bigl((Y_n)_{<3\rho_{\mathcal B}(t_n)+b_{\mathcal B}}\bigr)
 }{\mu(Y_n)}
 \longrightarrow0.
\]
Applying \eqref{eq:building-master-estimate} with \(t=t_n\), and using
\[
 t_n^{(1-r)/2}(\log t_n)^2\longrightarrow0
\]
because \(r\ge2\), proves \cref{intro:buildings}.
\end{proof}

\bibliographystyle{plain}
\bibliography{references_homology}

\end{document}